\documentclass[
final
]{dmtcs-episciences}

\usepackage{amssymb,amsmath,amsthm}
\usepackage{color,enumerate}
\usepackage{relsize}
\usepackage[utf8]{inputenc}
\usepackage{subfigure}
\usepackage{graphicx}
\usepackage{pgf,tikz}
\usepackage{mathrsfs}
\usepackage[round]{natbib}
\usetikzlibrary{arrows}
\usetikzlibrary{decorations.markings}
\newtheorem{theorem}{Theorem}[section]
\newtheorem{lemma}[theorem]{Lemma}
\newtheorem{corollary}[theorem]{Corollary}
\newtheorem{proposition}[theorem]{Proposition}

\newtheorem{remark}[theorem]{Remark}

\author{Selim Bahad\i r\affiliationmark{1}
  \and Didem G\"{o}z\"{u}pek\affiliationmark{2}}
\title[Graphs with Large Total Domination Number]{On a Class of Graphs with Large Total Domination Number \thanks{This work is supported by the Scientific and Technological Research
Council of Turkey (TUBITAK) under grant no. 114E731.}}
\affiliation{
Department of Mathematics-Computer, Ankara Y\i ld\i r\i m Beyaz\i t University, Turkey\\
Department of Computer Engineering, Gebze Technical University, Turkey}
\received{2017-5-3}
\revised{2018-5-23}
\accepted{2018-5-25}
\begin{document}
\publicationdetails{20}{2018}{1}{23}{3877}
\maketitle
\begin{abstract}
Let $\gamma(G)$ and $\gamma_t(G)$ denote the domination number and the total domination number, respectively, of a graph $G$ with no isolated vertices.
It is well-known that $\gamma_t(G) \leq 2\gamma(G)$.
We provide a characterization of a large family of graphs (including chordal graphs) satisfying $\gamma_t(G)= 2\gamma(G)$,
strictly generalizing the results of Henning (2001) and Hou and Xu (2010),
and partially answering an open question of Henning (2009).
\end{abstract}

\section{Introduction}
\label{sec:intro}
Let $G$ be a simple graph with vertex set $V(G)$ and edge set $E(G)$.
The neighborhood of a vertex $v\in V(G)$, denoted by $N(v)$, is the set of vertices adjacent to $v$.
For any subset $S\subseteq V(G)$, the neighborhood of $S$ is $\cup _{v\in S}N(v)$ and is denoted by $N(S)$.
The closed neighborhood of a subset $S\subseteq V(G)$, denoted by $N[S]$, is $N(S)\cup S$.
In particular, the closed neighborhood of a vertex $v$ is denoted by $N[v]$.

A set $S\subseteq V(G)$ of vertices is called a \emph{dominating set} of $G$ if every vertex of $V(G)\backslash S$ is adjacent to a member of $S$, that is, $N[S]=V(G)$.
The \emph{domination number} $\gamma(G)$ is the minimum cardinality of a dominating set of $G$.
If $G$ has no isolated vertices, a subset $S\subseteq V(G)$ is called a \emph{total dominating set} of $G$ if every vertex of $V(G)$ is adjacent to a member of $S$, i.e., $N(S)=V(G)$.
In other words, $S$ is a total dominating set if $S$ is a dominating set and the subgraph of $G$ induced by $S$ has no isolated vertices.
The \emph{total domination number} of $G$ with no isolated vertices, denoted by $\gamma_t(G)$, is the minimum size of a total dominating set of $G$.
A minimum dominating set is called  a $\gamma$-set of $G$ and a minimum total dominating set is called a $\gamma_t$-set of $G$.

Obtaining bounds on total domination number in terms of other graph parameters and classifying graphs whose total domination number attains an upper or lower bound are studied by many authors (see, Chapter 2 in \cite{yeo2013}).
For example, \cite{cockayne1980} showed that if $G$ is a connected graph with order at least 3, then $\gamma_t(G)\leq 2|V(G)|/3$.
Moreover, \cite{brigham2000} proved that a connected graph $G$ satisfies
$\gamma_t(G)= 2|V(G)|/3$ if and only if $G$ is a cycle of length 3 or 6, or $H\circ P_2$ for some connected graph $H$,
where $P_2$ is a path of length 2 and $H\circ P_2$ is obtained by identifying each vertex of $H$ by an end vertex of a copy of $P_2$.

As every $\gamma_t$-set is a dominating set as well, we have $\gamma(G)\leq \gamma_t(G)$.
For any $\gamma$-set $S$, one can extend $S$ to a total dominating set of cardinality at most $2\gamma(G)$ by including a neighbor of each vertex of $S$,
and therefore we get $\gamma(G) \leq \gamma_t(G)\leq 2\gamma(G)$, which is first observed by \cite{bollobas1979}.
In this paper, motivated by an open problem in \cite{henning2009},
we study graphs satisfying the upper bound for total domination number, $\gamma_t(G)= 2\gamma(G)$, and
refer to them as $(\gamma_t,2\gamma)$-graphs.

\cite{henning2001} provided a constructive characterization of $(\gamma_t,2\gamma)$-trees, whereas
\cite{hou2010} generalized it to block graphs and gave a characterization of $(\gamma_t,2\gamma)$-block graphs.
We extend the results in \cite{hou2010} to a larger family of graphs and partially solve the open problem (characterizing all $(\gamma_t,2\gamma)$-graphs) in \cite{henning2009}.

The rest of this paper is organized as follows:
Section \ref{sec:mainres} provides the main theorem and its applications.
The proof of the main theorem is given in Section \ref{sec:proof}.
Section \ref{sec:relwork} presents previous results on $(\gamma_t,2\gamma)$-graphs and their verifications by using our main theorem.
Discussion and conclusions are provided in Section \ref{sec:dis}.

\section{Main Results}
\label{sec:mainres}

We first provide some definitions required for the statement of the main theorem.
Two vertices $u$ and $v$ in $G$ are called \emph{true twins} whenever $N[u]=N[v]$,
i.e., in a pair of vertices a vertex is a true twin of the other one if they have the same closed neighborhood.
For each vertex $v$, we partition $N[v]$ into three sets, namely $T(v), D(v)$ and $M(v)$.
$T(v)$ consists of $v$ and its true twins.
A neighbor $u$ of $v$ is in $D(v)$ if $N[u]$ is a proper subset of $N[v]$.
That is, $u\in D(v)$ if and only if $u$ and $v$ are adjacent, $u$ is not a true twin of $v$, and every neighbor of $u$ other than $v$ is also a neighbor of $v$.
All other neighbors of $v$ are in $M(v)$, i.e., a neighbor of $v$ is in $M(v)$ if and only if it has a neighbor which is not adjacent to $v$.

We say that a vertex $v$ is \emph{special} if there is no $u \in M(v)$ such that $D(v) \subseteq N(u)$.
Isolated vertices are considered to be non special.
%In other words, a vertex $v$ is special if and only if the set of vertices adjacent to all members of $D(v)$ is $T(v)$.
Note that if $D(v)=\emptyset$ and $M(v)\neq \emptyset$, then $v$ is not special.
Moreover, notice that if a vertex is special, then all of its true twins are special as well.
See Figure \ref{fig:specialex} for an example of a special vertex.

\begin{figure}\center
\begin{tikzpicture}[line cap=round,line join=round,>=triangle 45,x=1.07cm,y=1.07cm]
\clip(0.5,0.5) rectangle (5.5,3.0);
\draw (2.,2.5)-- (3.,2.5);
\draw (2.,2.5)-- (1.,1.);
\draw (2.,2.5)-- (2.,1.5);
\draw (2.,2.5)-- (3.,1.5);
\draw (2.,2.5)-- (4.,1.);
\draw (1.,1.)-- (2.,1.5);
\draw (2.,1.5)-- (3.,1.5);
\draw (3.,1.5)-- (4.,2.5);
\draw (4.,1.)-- (5.,2.5);
\draw (3.,2.5)-- (1.,1.);
\draw (3.,2.5)-- (2.,1.5);
\draw (3.,2.5)-- (3.,1.5);
\draw (3.,2.5)-- (4.,1.);
\draw (4.,2.5)-- (5.,2.5);
\draw (1.,1.)-- (4.,1.);
\draw (1.75,2.95) node[anchor=north west] {$v_1$};
\draw (2.75,2.95) node[anchor=north west] {$v_2$};
\draw (0.75,1.0) node[anchor=north west] {$v_3$};
\draw (1.75,1.5) node[anchor=north west] {$v_4$};
\draw (2.75,1.5) node[anchor=north west] {$v_5$};
\draw (3.75,1.0) node[anchor=north west] {$v_6$};
\draw (3.75,2.95) node[anchor=north west] {$v_7$};
\draw (4.75,2.95) node[anchor=north west] {$v_8$};
\begin{scriptsize}
\draw [fill=black] (2.,2.5) circle (1.5pt);
\draw [fill=black] (4.,2.5) circle (1.5pt);
\draw [fill=black] (3.,2.5) circle (1.5pt);
\draw [fill=black] (1.,1.) circle (1.5pt);
\draw [fill=black] (2.,1.5) circle (1.5pt);
\draw [fill=black] (3.,1.5) circle (1.5pt);
\draw [fill=black] (4.,1.) circle (1.5pt);
\draw [fill=black] (5.,2.5) circle (1.5pt);
\end{scriptsize}
\end{tikzpicture}
\caption{In the given graph,
$N[v_1]=\{v_1,v_2,v_3,v_4,v_5,v_6\}$, $T(v_1)=\{v_1,v_2\}$, $D(v_1)=\{v_3,v_4\}$ and $M(v_1)=\{v_5,v_6\}$.
Since none of $N(v_5)=\{v_1,v_2,v_4,v_7\}$ and $N(v_6)=\{v_1,v_2,v_3,v_8\}$ includes $D(v_1)=\{v_3,v_4\}$,
we see that $v_1$ is a special vertex.
Note also that $v_1$ and $v_2$ are the only special vertices in this graph.
}
\label{fig:specialex}
\end{figure}
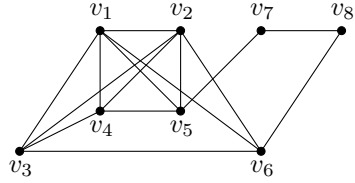

Partition the set of special vertices of $G$ in such a way that two vertices are in the same part if and only if they are true twins.
A set obtained by picking exactly one vertex from each part is called an \emph{$S(G)$-set}.
Hence, for any special vertex $v$, every $S(G)$-set contains exactly one element from $T(v)$.

A graph is called $(G_1,\dots, G_k)$\emph{-free} if it contains none of $G_1,\dots,G_k$ as an induced subgraph.
Let $H_1$ and $H_2$ be the graphs shown in Figure \ref{fig:H1H2}
and $C_k$ denote a cycle of length $k$.
A subset $S\subseteq V(G)$ is called a \emph{packing} in $G$ if $N[u]\cap N[v] =\emptyset$ for every distinct vertices $u$ and $v$ in $S$.

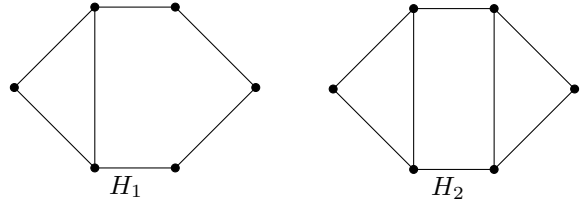
\begin{figure}\center
\begin{tikzpicture}[line cap=round,line join=round,>=triangle 45,x=1.07cm,y=1.07cm]
\clip(0.76,0.42) rectangle (8.38,3.4);
\draw (1.0385444216967172,2.0184934217695405)-- (2.0385444216967183,3.0184934217695405);
\draw (2.0385444216967183,3.0184934217695405)-- (3.0385444216967197,3.0184934217695405);
\draw (3.0385444216967197,3.0184934217695405)-- (4.03854442169672,2.0184934217695405);
\draw (4.03854442169672,2.0184934217695405)-- (3.0385444216967197,1.0184934217695407);
\draw (3.0385444216967197,1.0184934217695407)-- (2.0385444216967183,1.0184934217695407);
\draw (2.0385444216967183,1.0184934217695407)-- (1.0385444216967172,2.0184934217695405);
\draw (2.0385444216967183,1.0184934217695407)-- (2.0385444216967183,3.0184934217695405);
\draw (5.,2.)-- (6.,3.);
\draw (6.,3.)-- (7.,3.);
\draw (7.,3.)-- (8.,2.);
\draw (8.,2.)-- (7.,1.);
\draw (7.,1.)-- (6.,1.);
\draw (6.,1.)-- (5.,2.);
\draw (6.,1.)-- (6.,3.);
\draw (7.,1.)-- (7.,3.);
\draw (2.1,1.04) node[anchor=north west] {$H_1$};
\draw (6.1,1.02) node[anchor=north west] {$H_2$};
\begin{scriptsize}
\draw [fill=black] (1.0385444216967172,2.0184934217695405) circle (1.5pt);
\draw [fill=black] (2.0385444216967183,1.0184934217695407) circle (1.5pt);
\draw [fill=black] (2.0385444216967183,3.0184934217695405) circle (1.5pt);
\draw [fill=black] (3.0385444216967197,3.0184934217695405) circle (1.5pt);
\draw [fill=black] (3.0385444216967197,1.0184934217695407) circle (1.5pt);
\draw [fill=black] (4.03854442169672,2.0184934217695405) circle (1.5pt);
\draw [fill=black] (5.,2.) circle (1.5pt);
\draw [fill=black] (6.,3.) circle (1.5pt);
\draw [fill=black] (7.,3.) circle (1.5pt);
\draw [fill=black] (8.,2.) circle (1.5pt);
\draw [fill=black] (7.,1.) circle (1.5pt);
\draw [fill=black] (6.,1.) circle (1.5pt);
\end{scriptsize}
\end{tikzpicture}
\caption{The graphs $H_1$ and $H_2$.
}
\label{fig:H1H2}
\end{figure}

\begin{theorem}[Main Theorem]
\label{thm:main}
Let $G$ be an $(H_1,H_2,C_6)$-free graph and $S$ be an $S(G)$-set.
Then,
$G$ is a $(\gamma_t,2\gamma)$-graph if and only if $S$ is both a packing and a dominating set of $G$.
\end{theorem}

Since a chordal graph has no $H_1,H_2$ and $C_6$ as induced subgraphs,
we have the following result as a corollary of Theorem \ref{thm:main}.
\begin{corollary}\label{cor:chordal}
Let $G$ be a chordal graph and $S$ be an $S(G)$-set.
Then, $\gamma_t (G)=2\gamma(G)$ if and only if $S$ is a packing and a dominating set of $G$.
\end{corollary}
For chordal graphs,
both of the problems of finding the domination number and finding the total domination number are NP-complete
(see, \cite{booth1982} and \cite{laskar1984}, respectively.)
However,
constructing an $S(G)$-set and checking whether it is a packing and a dominating set can be easily done by an algorithm with polynomial time complexity.
Therefore, the problem of determining $\gamma_t (G)=2\gamma(G)$ for an $(H_1,H_2,C_6)$-free
(in particular, for a chordal graph) $G$ is polynomial time solvable.

By using the results of Theorem \ref{thm:main},
we next provide a characterization of another family of graphs $G$ with $\gamma_t(G)=2\gamma(G)$.
A  \emph{leaf} of a graph is a vertex with degree 1,
while a \emph{support vertex} of graph is a vertex adjacent to a leaf.
For the graph $K_2$, we assume that one vertex is leaf and the other one is support vertex.
Let $sup(G)$ denote the set of all support vertices in the graph $G$.

\begin{theorem}\label{thm:main2}
Let $G$ be a $(C_3,C_6)$-free graph.
Then, $\gamma_t (G)=2\gamma(G)$ if and only if $sup(G)$ is a packing and a dominating set of $G$.
\end{theorem}
\begin{proof}
As both $H_1$ and $H_2$ contain $C_3$ as a subgraph,
by Theorem \ref{thm:main} it suffices to show that $sup(G)$ is an $S(G)$-set.
When $G=K_2$, the claim is trivial.
Suppose $G\neq K_2$.
Let $v$ be a leaf of $G$ and $u$ be the support vertex adjacent to $v$.
Then, $D(v)=\emptyset$ and $M(v)=\{u\}$; therefore, $v$ is not special.
Now let $v$ be a vertex with at least two neighbors.
Then, since $G$ has no $C_3$, neighbors of $v$ form an independent set and $v$ has no true twin.
Therefore, $T(v)=\{v\}$ and any neighbor of $v$ is either a leaf or has a neighbor which is not adjacent to $v$.
Thus, $D(v)$ is the set of leaves adjacent to $v$.
If $D(v)=\emptyset$, then $v$ is not special since $M(v)\neq \emptyset$.
If $D(v)$ is nonempty (i.e., $v$ is a support vertex), then $v$ is special since $v$ is the unique neighbor of a leaf in $D(v)$.
Then, clearly we see that $v$ is special if and only if $v$ is a support vertex and therefore,
$sup(G)$ is the unique $S(G)$-set.
\end{proof}

The girth of a graph $G$, denoted by $g(G)$, is the length of a shortest cycle (if any) in $G$.
Acyclic graphs (forests) are considered to have infinite girth.
Since $sup(G)=\emptyset$ for a graph $G$ with minimum degree at least 2, Theorem \ref{thm:main2} implies the following result.
\begin{corollary}\label{cor:girth}
Let $G$ be a $(\gamma_t,2\gamma)$-graph with $\delta(G)\geq 2$.
Then $G$ contains an induced $C_3$ or $C_6$ and hence, $g(G)\leq 6$.
\end{corollary}

\section{Proof of the Main Theorem}\label{sec:proof}

We first present a simple but useful observation which is also partially given in \cite{hou2010}.

\begin{lemma}\label{lem:dompack}
Let $G$ be a $(\gamma_t,2\gamma)$-graph and $S$ be a subset of $V(G)$.
Then, $S$ is a $\gamma$-set of $G$ if and only if $S$ is a packing and a dominating set of $G$.
\end{lemma}
\begin{proof}
Let $S=\{v_1,\dots, v_k \}$ be a $\gamma$-set in $G$.
Then, note that $\gamma_t(G)=2\gamma(G)=2k$ and it suffices to show that $S$ is a packing.
Suppose that $S$ is not a packing in $G$.
Then, without loss of generality, we may assume that $N[v_1]$ and $N[v_2]$ have a common vertex, say $w$.
Let $w_i \in N(v_i)$ for $i=3,\dots , k$.
It is easy to verify that $S\cup \{w,w_3,\dots,w_k \}$ is a total dominating set of $G$, and therefore,
by definition of $\gamma_t$, we get $\gamma_t(G)\leq 2k-1$, which contradicts with $\gamma_t(G)=2k$.

Now let $S=\{v_1,\dots, v_k \}$ be both a packing and a dominating set of $G$.
Then, we have $\gamma(G)\leq k$ since $S$ is a dominating set.
On the other hand, every dominating set of $G$ contains at least one vertex in $N[v_i]$ for $i=1,\dots,k$ and hence,
we get $\gamma(G)\geq k$, which implies that $\gamma(G)= k$.
Therefore, $S$ is a $\gamma$-set.
\end{proof}
Therefore, in every minimum dominating set of a graph $G$ with $\gamma_t(G)=2\gamma(G)$, any pair of vertices are nonadjacent and have no common neighbor.

We next provide a sufficient condition on a graph $G$ to satisfy $\gamma_t(G)=2\gamma(G)$.
\begin{lemma}\label{lem:sgpacdom}
Let $G$ be a graph and $S$ be an $S(G)$-set.
If $S$ is both a packing and a dominating set of $G$, then $G$ is a $(\gamma_t,2\gamma)$-graph.
\end{lemma}
\begin{proof}
Suppose that the $S(G)$-set $S=\{v_1,\dots, v_k\}$ is a packing and a dominating set of $G$.
Then, $N[v_1],\dots, N[v_k]$ is a partition of $V(G)$.
Notice that every dominating set of $G$ should include at least one vertex from $N[v_i]$ and hence,
we get $\gamma(G)\geq k$.
On the other hand, $S$ is a dominating set with cardinality $k$ and therefore, $\gamma(G)=k$.

Let $A$ be a total dominating set of $G$.
Since $A$ dominates $v_i$, at least one member of $N(v_i)$ is in $A$ and hence, $|A\cap N[v_i]|\geq 1$ for $i=1,\dots,k$.
Assume that $A\cap N[v_i]=\{u\}$ for some $u$ and $i$.
If $u$ is in $T(v_i)\cup D(v_i)$, then there is no element in $A$ adjacent to $u$, contradiction.
Thus, $u$ is in $M(v_i)$.
Since $v_i$ is special, $D(v_i)$ is nonempty and $u$ is not adjacent to all the vertices in $D(v_i)$ and hence,
at least one vertex in $D(v_i)$ is not dominated by $A$, contradiction.
Consequently, we get $|A\cap N[v_i]|\geq 2$ for $i=1,\dots,k$ and therefore, $|A|\geq 2k$,
which implies that $\gamma_t(G)\geq 2k=2\gamma(G)$.
However, as $\gamma_t(G)\leq 2\gamma(G)$ for any graph with no isolated vertices, we obtain that $\gamma_t(G)= 2\gamma(G)$.
\end{proof}

However, the converse of the result in Lemma \ref{lem:sgpacdom} does not hold for every graph $G$ and hence,
extra conditions on $G$ are required to make the condition on $S$ necessary.
\begin{proposition}\label{prop:main}
Let $G$ be an $(H_1,H_2,C_6)$-free graph and $S$ be an $S(G)$-set.
If $\gamma_t(G)=2\gamma(G)$, then $S$ is both a packing and a dominating set of G.
\end{proposition}
\begin{proof}
Suppose that $\gamma_t(G)=2\gamma(G)$.
Let $\{v_1,\dots ,v_k \}$ be a $\gamma$-set of $G$ (so, $\gamma(G)=k$).
Then, by Lemma \ref{lem:dompack} we obtain that $N[v_1], \dots,N[v_k]$ is a partition of $V(G)$.
In addition, note that every edge between $N[v_i]$ and $N[v_j]$ has endpoints one from $M(v_i)$ and one from $M(v_j)$ for every $i\neq j$.

We show that the set of all special vertices is $\cup_{i=1}^k T(v_i)$.
Without loss of generality, it suffices to prove that the set of special vertices in $N[v_1]$ is $T[v_1]$.

We first show that $v_1$ is a special vertex.
Suppose that $v_1$ is not special.
Then, the set $M(v_1)$ is nonempty and contains a vertex $u$ such that $D(v_1)\subseteq N(u)$.
Notice that in the case of $D(v_1)=\emptyset$ any member of $M(v_1)$ can be chosen to be $u$.
Note also that existence of such a vertex $u$ implies that $k\geq 2$.
Let $m=\max_{x\in N(v_2)} |N(x)\cap N(v_1)|$.
Then, for some $w_2\in N(v_2)$ and $u_1,\dots,u_m \in N(v_1)$,
$w_2$ is adjacent to $u_1,\dots ,u_m$ and has no other neighbor in $N(v_1)$.
In other words, $w_2$ is a neighbor of $v_2$ sharing the maximum number of neighbors with $v_1$ among $N(v_2)$.
We claim that if $y\in N(v_1)$ shares a neighbor with $v_2$, then $y$ is adjacent to $w_2$.
Assume the contrary.
Then, $y\neq u_i$ for $i=1,\dots,m$ and $y$ is adjacent to some $z\in N(v_2)$ such that $z\neq w_2$.
Now consider the subgraph of $G$ induced by $v_1, y, z,v_2,w_2$ and $u_i$ for some fixed $i$
(see Figure \ref{fig:C6}).
\begin{figure}\center
\begin{tikzpicture}[line cap=round,line join=round,>=triangle 45,x=1.0cm,y=1.0cm]
\clip(0.5,0.4908699361412032) rectangle (4.8,3.566359722231211);
\draw (1.,2.)-- (2.,3.);
\draw (2.,3.)-- (3.,3.);
\draw (3.,3.)-- (4.,2.);
\draw (4.,2.)-- (3.,1.);
\draw (3.,1.)-- (2.,1.);
\draw (2.,1.)-- (1.,2.);
\draw [dash pattern=on 1pt off 1pt](2.,1.)-- (2.,3.);
\draw [dash pattern=on 1pt off 1pt](3.,1.)-- (3.,3.);
\draw [dash pattern=on 1pt off 1pt](2.,3.)-- (3.,1.);
\draw (0.4,2.25) node[anchor=north west] {$v_1$};
\draw (4,2.25) node[anchor=north west] {$v_2$};
\draw (1.75,0.95) node[anchor=north west] {$y$};
\draw (2.75,0.95) node[anchor=north west] {$z$};
\draw (1.75,3.5) node[anchor=north west] {$u_i$};
\draw (2.75,3.5) node[anchor=north west] {$w_2$};
\begin{scriptsize}
\draw [fill=black] (1.,2.) circle (1.5pt);
\draw [fill=black] (2.,1.) circle (1.5pt);
\draw [fill=black] (2.,3.) circle (1.5pt);
\draw [fill=black] (3.,3.) circle (1.5pt);
\draw [fill=black] (3.,1.) circle (1.5pt);
\draw [fill=black] (4.,2.) circle (1.5pt);
\end{scriptsize}
\end{tikzpicture}
\caption{The subgraph of $G$ induced by $\{v_1, y, z,v_2,w_2,u_i\}$.
Dashed edges represent possible edges in $G$.
}
\label{fig:C6}
\end{figure}
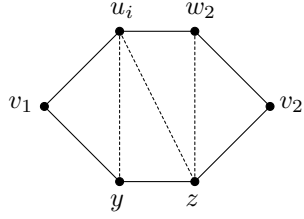
Note that in this graph $v_1, y, z,v_2,w$ and $u_i$ forms a $C_6$,
$v_1$ is adjacent to only $u_i$ and $y$, $v_2$ is adjacent to only $w_2$ and $z$,
and $y$ is not adjacent to $w_2$.
Suppose that $z$ is not adjacent to $u_i$.
Since $G$ has no induced $C_6$, $u_i$ and $y$ are adjacent or $w_2$ and $z$ are adjacent.
If exactly one of them occurs, then we have an induced $H_1$, and if both occurs, then we obtain an induced $H_2$.
Therefore, in any of the cases we get a contradiction and hence, $z$ is adjacent to $u_i$.
Thus, $z$ is adjacent to each $u_i$ together with $y$ and therefore, $z$ has at least $m+1$ neighbors in $N(v_1)$, which contradicts the maximality of $m$.
Thus, any element of $M(v_1)$ that is adjacent to a neighbor of $v_2$ is also adjacent to $w_2$.
Similarly, define $w_3\in N(v_3), \dots ,w_k \in N(v_k)$; i.e.,
$w_i$ is a vertex sharing the maximum number of neighbors with $v_1$ among the neighbors of $v_i$.
Then, any vertex in $M(v_1)$ which has a neighbor in $N(v_i)$, where $2\leq i \leq k$, is adjacent to $w_i$.
Note that by definition, every vertex in $M(v_1)$ has a neighbor in $N(v_i)$ for some $i\geq 2$, and therefore,
every element in $M(v_1)$ is adjacent to at least one vertex in $\{w_2,\dots, w_k \}$.
Then, the set $\{u,v_2,w_2,\dots, w_k,v_k \}$ is a total dominating set of $G$
because $u$ dominates $T(v_1)\cup D(v_1)$,
$\{w_2,\dots, w_k \}$ dominates $M(v_1)$, and
$\{w_i,v_i\}$ dominates $N[v_i]$ for $i=2,\dots,k$.
However, this total dominating set has cardinality $2k-1$, which contradicts the assumption that $\gamma_t(G) =2\gamma(G)$.
Therefore, we conclude that $v_1$ is special.

We next show that none of the vertices in $M(v_1)$ is special.
Let $u\in M(v_1)$.
For any $x\in N(v_i)$ with $i\geq 2$, we have $v_i\in N(x)$ and $v_i\notin N(u)$ and hence,
$x$ is not in $D(u)$.
Therefore, $D(u)$ is a subset of $N[v_1]\backslash \{u\}$.
Moreover, as $v_1$ is special, $u$ is not adjacent to at least one vertex of $D(v_1)$.
Thus, $v_1$ is not in $D(u)$ and so,
$D(u)$ is a proper subset of $N(v_1)$, which yields that $u$ is not special.

We then show that none of the vertices in $D(v_1)$ is special.
Let $u\in D(v_1)$.
Then, $N[u]$ is a proper subset of $N[v_1]$ and hence, $v_1$ is in neither $T(u)$ nor $D(u)$, that is, $v_1 \in M(u)$.
Therefore, $D(u) \subseteq N[u]\backslash \{v_1\} \subseteq N[v_1] \backslash \{v_1\} =N(v_1)$ and hence,
$u$ is not special.
Consequently, we obtain that all special vertices in $N[v_1]$ are members of $T(v_1)$.

By similar arguments for $v_2,\dots,v_k$, we see that the set of special vertices is $\cup_{i=1}^k T(v_i)$.
Then, recall that an $S(G)$-set $S$ is equal to $\{v_1',\dots,v_k' \}$, where $v_i' \in T(v_i)$ for $i=1,\dots,k$.
Since $\{v_1,\dots,v_k \}$ is both a packing and a dominating set of $G$, so is $S$ and hence, the desired result follows.
\end{proof}

By combining the results in Lemma \ref{lem:sgpacdom} and Proposition \ref{prop:main},
we obtain our main result given in Theorem \ref{thm:main}.

\begin{remark}\label{rem:G1G2}
Forbidden graphs $H_1,H_2$ and $C_6$ are best possible in the sense that
if one allows one of these three graphs, then the statement in Proposition \ref{prop:main} is no longer true.
For each case, we have the following counterexamples.
Let $G_1$ and $G_2$ be the graphs presented in Figure \ref{fig:G1G2}.
Clearly, $G_1$ has an induced $H_1$ and is $(H_2,C_6)$-free, $\gamma_t(G_1)=2\gamma(G_1)=4$; however, $v$ is the unique special vertex in $G_1$.
Similarly, $G_2$ contains an induced $H_2$ and is $(H_1,C_6)$-free, and it is easy to verify that $\gamma_t(G_2)=2\gamma(G_2)=6$;
however, $G_2$ has two special vertices, $v_1$ and $v_2$.
Finally, $C_6$ is $(H_1,H_2)$-free, $\gamma_t(C_6)=2\gamma(C_6)=4$; however, $C_6$ has no special vertices.
\end{remark}
\begin{figure}\center
\begin{tikzpicture}[line cap=round,line join=round,>=triangle 45,x=1.0cm,y=1.0cm]
\clip(-0.1,0.26) rectangle (14.18,3.48);
\draw (1.,2.)-- (2.,3.);
\draw (2.,3.)-- (3.,3.);
\draw (3.,3.)-- (4.,2.);
\draw (4.,2.)-- (3.,1.);
\draw (3.,1.)-- (2.,1.);
\draw (2.,1.)-- (1.,2.);
\draw (2.,1.)-- (2.,3.);
\draw (0.,2.)-- (1.,2.);
\draw (6.,2.)-- (7.,2.);
\draw (7.,2.)-- (8.,3.);
\draw (8.,3.)-- (9.,3.);
\draw (9.,3.)-- (10.,2.);
\draw (10.,2.)-- (9.,1.);
\draw (9.,1.)-- (8.,1.);
\draw (8.,1.)-- (7.,2.);
\draw (8.,3.)-- (8.,1.);
\draw (9.,3.)-- (9.,1.);
\draw (10.,2.)-- (11.,3.);
\draw (11.,3.)-- (11.,1.);
\draw (11.,1.)-- (10.,2.);
\draw (11.,3.)-- (12.,3.);
\draw (12.,3.)-- (12.,1.);
\draw (12.,1.)-- (11.,1.);
\draw (12.,1.)-- (13.,2.);
\draw (13.,2.)-- (14.,2.);
\draw (12.,3.)-- (13.,2.);
\draw (0.7,2.5) node[anchor=north west] {$v$};
\draw (6.6,2.5) node[anchor=north west] {$v_1$};
\draw (12.75,2.5) node[anchor=north west] {$v_2$};
\draw (1.65,0.9) node[anchor=north west] {$G_1$};
\draw (9.65,0.88) node[anchor=north west] {$G_2$};
\begin{scriptsize}
\draw [fill=black] (1.,2.) circle (1.5pt);
\draw [fill=black] (2.,1.) circle (1.5pt);
\draw [fill=black] (2.,3.) circle (1.5pt);
\draw [fill=black] (3.,3.) circle (1.5pt);
\draw [fill=black] (3.,1.) circle (1.5pt);
\draw [fill=black] (4.,2.) circle (1.5pt);
\draw [fill=black] (0.,2.) circle (1.5pt);
\draw [fill=black] (6.,2.) circle (1.5pt);
\draw [fill=black] (7.,2.) circle (1.5pt);
\draw [fill=black] (8.,3.) circle (1.5pt);
\draw [fill=black] (8.,1.) circle (1.5pt);
\draw [fill=black] (9.,1.) circle (1.5pt);
\draw [fill=black] (9.,3.) circle (1.5pt);
\draw [fill=black] (10.,2.) circle (1.5pt);
\draw [fill=black] (11.,3.) circle (1.5pt);
\draw [fill=black] (11.,1.) circle (1.5pt);
\draw [fill=black] (12.,1.) circle (1.5pt);
\draw [fill=black] (12.,3.) circle (1.5pt);
\draw [fill=black] (13.,2.) circle (1.5pt);
\draw [fill=black] (14.,2.) circle (1.5pt);
\end{scriptsize}
\end{tikzpicture}
\caption{The graphs $G_1$ and $G_2$.
Notice that $v$ is the unique special vertex of $G_1$ and the special vertices in $G_2$ are only $v_1$ and $v_2$.}
\label{fig:G1G2}
\end{figure}
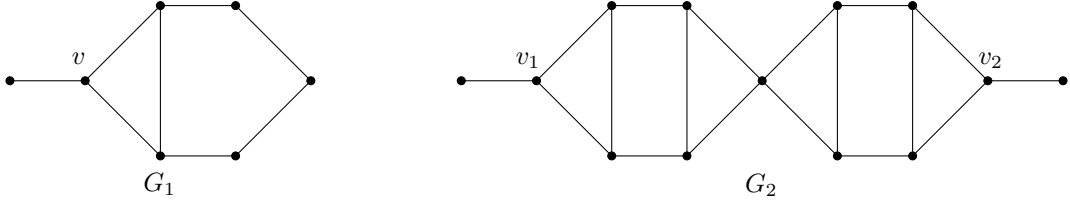

In the proof of Proposition \ref{prop:main} we actually show that if
$G$ contains no $H_1, H_2$ or $C_6$ as induced subgraphs and satisfies $\gamma_t(G)=2\gamma(G)$,
then every $\gamma$-set of $G$ is an $S(G)$-set and every $S(G)$-set is both a packing and a dominating set of $G$.
Combining this result with the one in Lemma \ref{lem:dompack} yields the following corollary.
\begin{corollary}\label{cor:nummindomset}
Let $G$ be an $(H_1,H_2,C_6)$-free $(\gamma_t,2\gamma)$-graph and $S$ be an $S(G)$-set.
Then, the number of $\gamma$-sets of $G$ is the number of $S(G)$-sets, which is
\begin{align*}
\prod_{v\in S} |T(v)|.
\end{align*}
In particular, $G$ has a unique $\gamma$-set if and only if no special vertex has a true twin.
\end{corollary}

\section{Related Work}\label{sec:relwork}

\cite{henning2001} provided a characterization of $(\gamma_t,2\gamma)$-trees,
which was then updated to the following result by \cite{yeo2013}.
\begin{theorem}[Theorem 4.8 in \cite{yeo2013}]\label{thm:henningtrees}
A tree $T$ of order at least 3 satisfies $\gamma_t(T)=2\gamma(T)$ if and only if
$T$ has a dominating set $S$ such that the following two conditions hold:\\
(a) Every vertex of $S$ is a support vertex of $T$.\\
(b) The set $S$ is a packing in $T$.
\end{theorem}
Note that a subset $S \subseteq sup(T)$ is a dominating set of $T$ if and only if $S=sup(T)$.
Therefore, Theorem \ref{thm:henningtrees} can be restated as follows:
$T$ is a $(\gamma_t,2\gamma)$-graph if and only if
$sup(T)$ is a dominating set and a packing of $T$, which is a corollary of Theorem \ref{thm:main2} since every
tree is $(C_3,C_6)$-free.

We next present another classification of $(\gamma_t,2\gamma)$-graphs, provided by \cite{hou2010}, and show that it is a corollary of Theorem \ref{thm:main}.
Let $G$ be a connected graph with no isolated vertex.
A vertex $v$ in $G$ is called a \emph{cut-vertex} if removing $v$ from the graph produces a disconnected graph.
A maximal connected subgraph with no cut-vertex is called a \emph{block}.
Note that a block of $G$ is either a maximal 2-connected subgraph of $G$ or a $K_2$.
The graph $G$ is called a \emph{block graph} if every block of $G$ is a complete graph.

The characterization in \cite{hou2010} is based on two sets $D_1$ and $D_2$.
$D_1$ is defined to be the set of cut-vertices which constitute a unique cut-vertex in a block of $G$, and
$D_2$ denotes the set of cut-vertices that have at least two neighbors which are not cut-vertices and belong to different blocks of $G$.
\begin{theorem}[Theorem 4 in \cite{hou2010}] \label{thm:hou}
Let $G$ be a connected block graph with at least two blocks.
Then $\gamma_t(G)=2\gamma(G)$ if and only if $G$ satisfies the following three conditions:\\
(1) $G$ has a unique minimum dominating set $D$,\\
(2) $D=D_1 \cup D_2$,\\
(3) $D$ is a packing in $G$.
\end{theorem}
We now show that $D$ is the unique $S(G)$-set, and hence, Theorem \ref{thm:main} implies Theorem \ref{thm:hou}.
First note that a vertex is a cut-vertex if and only if it belongs to at least two blocks.
Furthermore, since $G$ is connected and has at least two blocks, every block of $G$ has at least one cut-vertex.
Moreover, if $v$ is not a cut-vertex and $u$ is a cut-vertex in the block containing $v$,
then $N[v]$ is a proper subset of $N[u]$ since every block in $G$ is a clique.
Therefore, any special vertex is a cut-vertex.
Now let $v$ be a cut-vertex and $B_1,\dots, B_k$ be the blocks containing $v$.
Let $A_i$ be the set of vertices in $B_i$ which are not cut-vertices of $G$ for $i=1,\dots, k$.
It is easy to see that $v$ has no true twin and $D(v)=\cup_{i=1}^k A_i$.
If each $A_i$ is empty, then clearly $v$ is not special.
If only one $A_i$ is nonempty, suppose $A_1\neq \emptyset$.
Then $v$ is special if and only if $v$ is the unique cut-vertex of $B_1$, i.e.,
in such a case, $v$ is special if and only if $v\in D_1 \backslash D_2$.
Finally, if at least two of $A_i$s are nonempty (i.e., $v\in D_2$), then clearly $v$ is the unique vertex adjacent to every vertex in $D(v)$ and hence,
$v$ is special.
Therefore, the set of special vertices is $D=D_1 \cup D_2$ and it is the unique $S(G)$-set.
Consequently, Theorem \ref{thm:hou} is a corollary of Theorem \ref{thm:main}.
Besides, Theorem \ref{thm:main} implies that in Theorem \ref{thm:hou},
the condition on $D$ being the unique minimum dominating set of $G$ can be replaced by being a dominating set of $G$.

\section{Discussion and Conclusions}\label{sec:dis}
\cite{henning2009} presents a list of top fundamental problems on total domination in graphs.
Motivated by the problem of characterizing the graphs $G$ satisfying $\gamma_t(G)=2\gamma(G)$ in his list,
we study $(H_1,H_2,C_6)$-free graphs and provide a necessary and sufficient condition on them to be $(\gamma_t,2\gamma)$.
Thus, we extend the previous results on $(\gamma_t,2\gamma)$-graphs
(particularly, trees in \cite{henning2001} and block graphs in \cite{hou2010})
to a larger family of graphs including chordal graphs and $(C_3,C_6)$-free graphs.

In Lemma \ref{lem:sgpacdom}, we provide a sufficient condition for a graph to be a $(\gamma_t,2\gamma)$-graph.
We show that for any graph $G$ having an $S(G)$-set which is both a packing and a dominating set
forces $G$ to satisfy $\gamma_t(G)=2\gamma(G)$.
It also enables one to construct $(\gamma_t,2\gamma)$-graphs.
For example, let $G_1,\dots ,G_n$ and $G$ be pairwise disjoint arbitrary graphs such that $G$ has order $n$.
Let $V(G)=\{v_1,\dots, v_n\}$ and $u_1,\dots,u_n$ be new vertices.
Consider the graph $H$ obtained by connecting $u_i$ to $v_i$ and every vertex in $G_i$ for $i=1,\dots,n$.
It is easy to verify that the set of special vertices of $H$ is $\{u_1,\dots, u_n\}$,
which is both a dominating set and a packing of $H$, and hence, $H$ is a $(\gamma_t,2\gamma)$-graph.
Notice that if each $G_i$ is a singleton, then $H=G\circ P_2$.
By Corollary \ref{cor:girth}, we see that if the minimum degree of a $(\gamma_t,2\gamma)$-graph is 2, then its girth is at most 6.
However, since $g(G\circ P_2)=g(G)$, there is no bound on the girth of a $(\gamma_t,2\gamma)$-graph containing a leaf.

A $(\gamma_t,2\gamma)$-graph does not have to have special vertices (e.g., $C_6$).
However, by forbidding the graphs $H_1,H_2$ and $C_6$ in $G$, we see that the sufficient condition on $G$ to be a $(\gamma_t,2\gamma)$-graph is also necessary.
Thus, we obtain a characterization of $(H_1,H_2,C_6)$-free $(\gamma_t,2\gamma)$-graphs,
which allows one to solve the problem of determining whether a given $(H_1,H_2,C_6)$-free graph is $(\gamma_t,2\gamma)$ or not in polynomial time.
Recall also that forbidden graphs $H_1,H_2$ and $C_6$ are best possible in the sense that
allowing one of these three graphs avoids our main theorem being valid (see, Remark \ref{rem:G1G2}).

Notice that since the graphs $H_1$ and $H_2$ have triangles, we have the characterization of $C_6$-free $(\gamma_t,2\gamma)$-bipartite graphs.
Classifying all $(\gamma_t,2\gamma)$-bipartite graphs is a topic of ongoing research.
Another potential research direction is to investigate properties of $(\gamma_t,2\gamma)$-graphs in terms of other graph parameters.

\nocite{*}
\bibliographystyle{abbrvnat}
% use the following instead if you encounter problems
%\bibliographystyle{alpha}
%\begin{sloppypar}
\bibliography{references}{}
%\end{sloppypar}
\label{sec:biblio}

\end{document}